\documentclass[11pt, bezier]{article}
\usepackage{amsmath, amssymb, amsfonts, euscript, pifont, mathrsfs, latexsym, graphicx, yfonts, float}

\newtheorem{prethm}{{\bf Theorem}}

\newenvironment{thm}{\begin{prethm}\sl{\hspace{-0.5
               em}{\bf.}}}{\end{prethm}}

\newtheorem{prepro}[prethm]{{\bf Proposition}}

\newtheorem{prelem}[prethm]{{\bf Lemma}}

\newenvironment{lem}{\begin{prelem}\sl{\hspace{-0.5
               em}{\bf.}}}{\end{prelem}}

\newtheorem{predeff}[prethm]{{\bf Definition}}

\newenvironment{deff}{\begin{predeff}\rm{\hspace{-0.5
               em}{\bf.}}}{\end{predeff}}

\newtheorem{precor}[prethm]{{\bf Corollary}}

\newenvironment{cor}{\begin{precor}\sl{\hspace{-0.5
               em}{\bf.}}}{\end{precor}}

\newtheorem{preconj}[prethm]{{\bf Conjecture}}

\newtheorem{preremark}[prethm]{{\bf Remark}}

\newenvironment{remark}{\begin{preremark}\rm{\hspace{-0.5
               em}{\bf.}}}{\end{preremark}}

\newtheorem{preexample}[prethm]{{\bf Example}}

\newtheorem{preproof}{{\bf\textsf{Proof.}}}

\newenvironment{proof}[1]{\begin{preproof}{\rm
               #1}\hfill{$\Box$}}{\end{preproof}}

\DeclareMathAlphabet{\mathpzc}{OT1}{pzc}{m}{it}

\title{\bf\LARGE   Integral trees with given nullity }

\author{\large E. Ghorbani$^{\,\rm 1, 2}$ \quad \quad  A. Mohammadian$^{\,\rm 2}$ \quad \quad B. Tayfeh-Rezaie$^{\,\rm 2}$\\[.4cm]
{\sl $^{\rm 1}$Department of Mathematics, K.N. Toosi University of Technology,}\\
{\sl P.O. Box 16315-1618, Tehran, Iran}\\[0.3cm]
{\sl $^{\rm 2}$School of Mathematics, Institute for Research in Fundamental
Sciences\,(IPM),}\\{\sl P.O. Box
19395-5746, Tehran, Iran }
\\[0.5cm]{
$\mathsf{e\_ghorbani@ipm.ir}$ \quad\quad  $\mathsf{ali\_m@ipm.ir}$ \quad\quad  $\mathsf{tayfeh}$-$\mathsf{r@ipm.ir}$}}

\date{}

\begin{document}
\maketitle

\vspace{5mm}

\begin{abstract}

\noindent A graph is called integral if all eigenvalues of its adjacency matrix consist entirely of integers.
We prove that for a given nullity more than  $1$, there are only  finitely many  integral trees.
Integral trees with nullity at most $1$ was already characterized by Watanabe and   Brouwer.
It is  shown that integral trees with nullity $2$ and $3$ are unique.

\vspace{5mm}
\noindent {\bf AMS Mathematics Subject Classification\,(2010):}   05C50,  05C05, 15A18.
\end{abstract}

\vspace{5mm}

\section{Introduction}

For a graph $G$, we denote by $V(G)$, the vertex set of $G$ and the {\sl order} of $G$ is defined as $|V(G)|$.
The {\sl adjacency matrix}  of   $G$, denoted by  $A(G)$,  has  its   rows and
columns  indexed by  $V(G)$  and its     $(u, v)$-entry  is $1$ if the vertices $u$ and
$v$ are adjacent and $0$ otherwise.
The   {\sl characteristic polynomial} of  $G$,   denoted by $\varphi(G; x)$, is
the characteristic polynomial of $A(G)$.  We will drop the indeterminate $x$ for the simplicity of notation. The zeros of $\varphi(G)$
are called the {\sl eigenvalues} of $G$. Note that  $A(G)$ is a real
symmetric matrix so that all eigenvalues of $G$ are real numbers.
We denote the eigenvalues of $G$ in non-increasing  order
as $\lambda_1(G)\geqslant\cdots\geqslant\lambda_n(G)$, where $n=|V(G)|$.
The graph $G$ is  said to be {\sl integral} if all  eigenvalues of $G$  are integers.
The {\sl nullity} of  $G$ is defined as  the nullity of $A(G)$, which is
equal to the multiplicity of $0$ as an   eigenvalue  of $G$. Quite a few number of articles  on nullity of graphs   have been published. We refer the reader to see \cite{gb}   and   references therein  for a survey on this topic.

The notion of integral graphs was
first introduced in   \cite{har}.
A lot of articles deal with integral graphs. We refer the reader to \cite{BCRS} for  a comprehensive  but rather old survey on the subject.
Here, we are concerned  with   integral trees.  These objects are  extremely
rare and hence  very  difficult to find.
For a long time,    it was  an open question  whether there exist integral
trees with arbitrarily large diameter   \cite{watsc}.
Recently,    this question was  affirmatively answered in    \cite{csi, gmt},  where the authors constructed  integral trees for any diameter.
It is well known that
the tree on two vertices is the only integral tree with nullity zero  \cite{wata0}.
Thereafter,  Brouwer   proved that any  integral  tree  with nullity $1$ is a
subdivision of a star graph where the order of the star graph  is a perfect square \cite{bro}. The latter  result
has motivated us to investigate integral trees from the `nullity' point of view.

In this article, we prove that with a fixed nullity more than $1$, there are only finitely many
integral trees.
We also characterize  integral trees with nullity $2$ and $3$  showing that
there is a unique integral tree with  nullity $2$ as well as  a unique integral tree with  nullity $3$.

\section{Reduced trees}

In this section we introduce `reduced trees' and derive some properties of their spectrum.
We shall use these properties in the next section to prove our finiteness result.

We denote  the  multiplicity of $\lambda$ as an eigenvalue of  a graph $G$ by ${\rm mult}(G; \lambda)$.
We  also denote  the number of eigenvalues of  $G$ in the interval $(-1, 1)$ by $m(G)$.
Write $P_n$ for the path graph of order  $n$.
For a vertex $v$ of  a  graph  $G$, we say that
there are {\sl $k$  pendant  $P_2$ at  $v$} if removing $v$ from $G$ increases the number of $P_2$ components by $k$.
A graph $G$ is called {\sl reduced} if there exists at most one pendant  $P_2$ at each vertex of $G$.

The following folklore fact,     which  is stated   in \cite[p.  49]{crs} as an exercise,  shows  that the reduced graph obtained from a  graph $G$ by removing some pendant $P_2$
has the same nullity as $G$.

\begin{lem}\label{=}
Let $G$ be a graph and   $v\in V(G)$ be of  degree $1$. If $u$ is the unique neighbor of $v$, then the nullities of $G$ and $G-\{u, v\}$ are the same.
\end{lem}

The  following result      is immediately deduced from  Lemma \ref{=} and is proved  in \cite[Theorem 2]{cve}.

\begin{cor}\label{pm}
The size of the maximum matching in a  tree  of order $n$ with nullity $h$ is $\tfrac{n-h}{2}$.
\end{cor}

The first and second statements of the following  theorem    are  respectively  obtained from the Cauchy interlacing theorem for symmetric matrices \cite[Corollary 2.5.2]{bh} and   the Perron--Frobenius  theory of nonnegative matrices \cite[Theorem 2.2.1]{bh}.

\begin{thm}\label{int}
If $G$ is a graph of order $n$ and $H$ is an induced  subgraph of $G$ of order  $m$, then $\lambda_{n-m+i}(G)\leqslant\lambda_i(H)\leqslant\lambda_i(G)$ for $i=1, \ldots, m$.
Moreover, if $G$ is a connected graph and $G\neq H$, then $\lambda_1(H)<\lambda_1(G)$.
\end{thm}

As a consequence of Theorem  \ref{int}, one readily deduces that $\lambda_1(G)>\lambda_2(G)$ for any connected graph $G$  of order at least  $2$.

\begin{lem}\label{delp2asl}
Let $G$ be a graph and   $v\in V(G)$ be of  degree $1$.  If $u$ is the unique neighbor of $v$, then  $m(G-\{u, v\})\leqslant m(G)$.
\end{lem}

\begin{proof}{
Note that $m(G-\{u, v\})=m(G-u)-1$. Applying  Theorem \ref{int} for  $G$ and $G-u$, we see  that $m(G-u)-1\leqslant m(G)$, implying  the result.
}\end{proof}

The following lemma  generalizes a result in \cite{wata0}.

\begin{lem}\label{m0} The tree $P_2$ is the only tree  with no eigenvalue in $(-1, 1)$.
\end{lem}

\begin{proof}{We have $m(P_1)=1$.
By induction on $n$, we will  show  for any tree $T$ of order $n\geqslant3$ that  $m(T)\geqslant1$.
Let     $v$ be  a vertex of  degree $1$ in a tree $T$   and $v'$ be its neighbor.
If $T_v=T-\{v, v'\}$ has a connected  component other than $P_2$, then it follows from    Lemma \ref{delp2asl}, $m(P_1)=1$, and   the induction hypothesis that  $m(T)\geqslant m(T_v)\geqslant1$, as desired.
Otherwise, all the connected components of $T_v$ must be $P_2$. Indeed, we may  assume that this property  holds for each pendant vertex $v$ of $T$. This  forces that $T=P_4$.  But $m(P_4)=2$ by  \cite[Table 2]{cvet}, completing the proof.
}\end{proof}

\begin{thm}\label{finitered}
For any nonnegative integer $k$, there are  finitely many reduced  trees with exactly $k$ eigenvalues in  $(-1, 1)$.
\end{thm}

\begin{proof}{We prove the assertion by  induction on $k$.
By Lemma \ref{m0}, we may assume  that $k\geqslant1$.
Let $T$ be a reduced tree of order $n$ and with $m(T)=k$.
First suppose  that there exists  $v\in V(T)$ such  that three of the connected components  $T_1,\ldots, T_d$ of $T-v$ are   not $P_2$. From   Theorem \ref{int}, $m(T-v)\leqslant k+1$.
Since $T$ is reduced, at most one of  $T_1, \ldots, T_d$ is $P_2$.
Hence,   Lemma \ref{m0} yields that   $d-1\leqslant\sum_{i=1}^dm(T_i)\leqslant k+1$ and $m(T_i)+2\leqslant\sum_{i=1}^dm(T_i)\leqslant k+1$ for   $i=1, \ldots, d$.
It follows that $d\leqslant k+2$ and $m(T_i)\leqslant k-1$ for $i=1,\ldots,d$.
Note that if  some  $T_i$ is  not  reduced, then   it has exactly  one  vertex with more than one pendant $P_2$ and such a vertex has exactly two pendant
$P_2$. So the assertion follows by the induction hypothesis.
Now  suppose otherwise. This  means that  any vertex of $T$ is of degree at most $3$ and
all vertices of degree $3$ of  $T$ have a pendant $P_2$.  Hence,  $T$ is   obtained from  a path graph  $P_t$   by attaching one  pendant $P_2$ at some vertices of degree $2$ in $P_t$, implying    $n\leqslant3t-4$. Moreover,   it follows from   Lemma \ref{delp2asl} that   $m(P_t)\leqslant m(T)$.
We know from  \cite[p.  9]{bh} that    $\lambda_i(P_t)=2\cos\tfrac{\pi\ell}{t+1}$ for  $\ell=1, \ldots, t$. Therefore, $m(P_t)\geqslant\tfrac{t-2}{3}$ and so    $t\leqslant3k+2$ which in turn implies that   $n\leqslant9k+2$. This completes the proof.
}\end{proof}

For later use, we need the following refinement of Lemma \ref{delp2asl}.

\begin{lem}\label{delp2}
Let $T$ be a tree with at least one   pendant $P_2$ at $v\in V(T)$. Then increasing the number of pendant $P_2$ at  $v$ by one, leaves the number of eigenvalues in $(-1,1)$ unchanged and increases the multiplicity of $1$ by one.
\end{lem}

\begin{proof}{
Suppose that $T'$ is the resulting tree from $T$ by adding  two   new vertices $a$ and $b$ where  $a$ is joined to both $b$ and $v$.
Let $c$ and $d$ be the   vertices of a pendent $P_2$ of $T$ at $v$.
Let $k={\rm mult}(T; 1)$ and  assume that  $\{x_1, \ldots, x_k\}$ is  a  basis for the eigenspace  $\mathscr{E}$ of $T$ corresponding to  eigenvalue $1$ when    $k\geqslant1$.  Since each  vector  $x\in\mathscr{E}$ takes the same value on $c$ and $d$, we conclude that
$x$ vanishes on $v$.  For each  $i$ with  $1\leqslant i\leqslant k$, extend $x_i$ to find  the vector  $y_i$   defined on  $V(T')$     with value $0$ on $\{a, b\}$. Also,  define the  vector $y_{k+1}$ so   that $y_{k+1}(a)=y_{k+1}(b)=1$,   $y_{k+1}(c)=y_{k+1}(d)=-1$,   and $0$ elsewhere.
It is now readily  verified    that  $\{y_1, \ldots, y_{k+1}\}$ is   an independent subset of the  eigenspace   of $T'$ corresponding to  eigenvalue $1$.
By    Theorem \ref{int}, ${\rm mult}(T'; 1)-1\leqslant{\rm mult}(T'-a; 1)=k$.
Hence,   ${\rm mult}(T'; 1)=k+1$, as desired.
Furthermore,  from  ${\rm mult}(T'; 1)={\rm mult}(T'-a; 1)+1$ and   by  applying   Theorem \ref{int} for  $T'$ and $T'-a$, one  concludes  that $m(T'-a)=m(T')+1$.
Since $m(T'-a)=m(T)+1$, we deduce  that  $m(T)=m(T')$. This completes  the proof.
}\end{proof}

\section{Finiteness of integral trees with a given nullity}

In this section we present our main result which states that for every integer $h\geqslant2$, there are  finitely many integral   trees with nullity  $h$.

\begin{deff} \label{stype}
By considering  a tree $T$ as  a connected bipartite graph, one finds   a unique  pair $\{A, B\}$ for which  $A$ and $B$ are  disjoint independent subsets of $T$  with  $V(T)= A\cup B$ and $B\neq\varnothing$.  Define $\mathcal{S}(T; A)$ to be  the tree obtained from $T$ by attaching a pendant vertex to each vertex
in $B$.
\end{deff}

\begin{deff}
Let $T$ be a tree.  For every distinct  vertices  $v_1, \ldots, v_k\in V(T)$  and  nonnegative integers  $s_1, \ldots, s_k$, we denote by
$T(v_1, \ldots, v_k; s_1, \ldots, s_k)$   the resulting tree from  $T$  by  attaching  $s_i$  pendant    $P_2$ at $v_i$ for $i=1, \ldots,   k$.
\end{deff}

\begin{remark}\label{weyl}
Let $T$ be a tree of order $n$ and $k\geqslant1$. For every distinct  vertices  $v_1, \ldots, v_k\in V(T)$  and  nonnegative integers  $s_1, \ldots, s_k$ with    $s_1\geqslant \cdots  \geqslant s_k$,  employing  the Courant--Weyl inequalities   \cite[Theorem  2.8.1]{bh} yields  that
$$
\lambda_i\big(T(v_1, \ldots, v_k; s_1, \ldots, s_k)\big)\geqslant \sqrt{s_i+1}+\lambda_n(T),
$$
for $i=1, \ldots,   k$.  Therefore,  when all  values  $s_1, \ldots, s_k$ go to infinity, then the $k$  largest eigenvalues of   $T(v_1, \ldots, v_k; s_1, \ldots, s_k)$ tend to infinity.
\end{remark}

\begin{deff} \label{defs}
Let  $p$, $q$, and $r$ be   nonnegative integers and let $T_1$, $T_2$, $T_3$, and $T_4$ be the trees are depicted in Figure \ref{def} with  some specified  vertices. We will  denote by $S(p)$, $S(p, q)$, $S(p, q, r)$, and   $S'(p, q, r)$ the trees  $T_1(u; p)$, $T_2(u, v; p, q)$, $T_3(u, v, w; p, q, r)$, and  $T_4(u, v, w; p, q, r)$, respectively.
\end{deff}
\begin{figure}[H]
\centering
\includegraphics[width=0.75\textwidth]{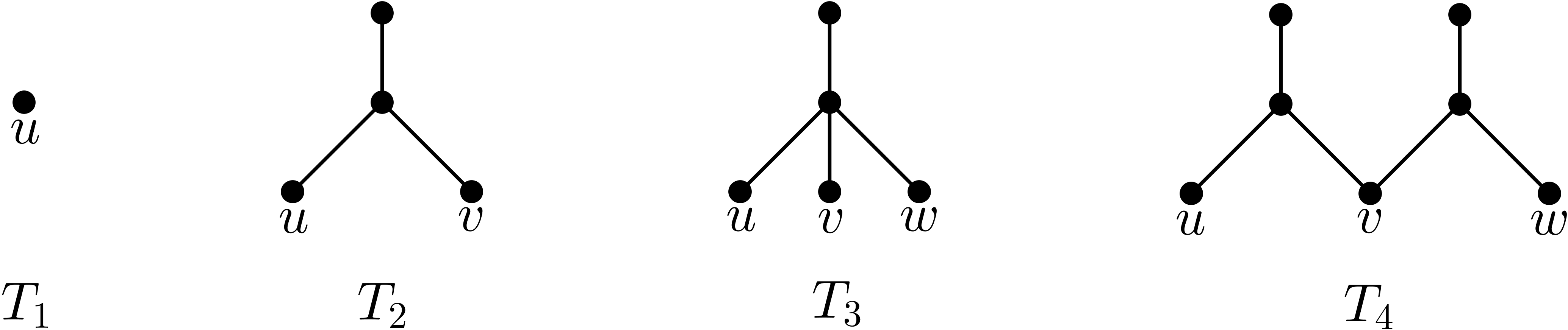}
\caption{\small{The trees of Deffnition \ref{defs}.}}\label{def}
\end{figure}

Note that all    trees introduced in  Definition \ref{defs} are of the form described in Definition \ref{stype}.
In the rest of the article, we will  use frequently the  next lemma which is proved in \cite[Lemma  2.8]{csi}.

\begin{lem}\label{cskvarithm}
Let $G$ be a  bipartite graph  with   bipartition $\{X, Y\}$ and  with $k$ positive eigenvalues. Let $G'$ be the graph  obtained from $G$ by joining $r$ new vertices of degree $1$ to each vertex  of $X$, for some positive integer $r$. Then $\lambda^2_i(G')=\lambda^2_i(G)+r$ for $i=1, \ldots, k$.
\end{lem}

The following lemma is established  in   \cite[Proposition 5.1.1(i)]{bh} for $k=1$. The general case is straightforwardly proved by induction on $k$ as mentioned in   \cite[p.  90]{bh}.

\begin{lem}\label{join}
Let   $T_1$ and   $T_2$ be two vertex disjoint   trees with  specified vertices   $v_1\in V(T_1)$ and  $v_2\in V(T_2)$.   For a positive integer  $k$, assume  that  $T$  is the tree  obtained  from
$T_1$ and $k$ copies of $T_2$  by joining  $v_1$ to the  $k$ copies of $v_2$. Then
$$\varphi(T)=\varphi(T_2)^{k-1}\big(\varphi(T_1)\varphi(T_2)-k\varphi(T_1-v_1)\varphi(T_2-v_2)\big).$$
\end{lem}

Using Lemma \ref{join}, one obtains  that
\begin{equation}\label{sp}\varphi\big(S(p)\big)=x(x^2-p-1)(x^2-1)^{p-1},
\end{equation}
for any nonnegative integer  $p$.

\begin{lem}\label{v1v2vk}
Let $T$ be a tree,   $k, t$ be positive integers,   and  $v_1, \ldots, v_k$ be distinct  vertices of $T$.
Suppose  that  there exists  a polynomial $f(x)$ such that   for  every  integers $s_1, \ldots, s_k\geqslant t$,  the tree  $T'=T(v_1, \ldots, v_k; s_1, \ldots, s_k)$ satisfies
\begin{equation}\label{ass}
\varphi(T')=(x^2-1)^{s_1+\cdots+s_k-k}f(x)\prod_{i=1}^k\big(x^2-\alpha_i(s_1, \ldots, s_k)\big),
\end{equation}
where  $\alpha_i(s_1, \ldots, s_k),$  $\ldots,$  $\alpha_k(s_1, \ldots, s_k)$  are  positive-valued functions in terms  of $s_1, \ldots, s_k$.
Then $T=\mathcal{S}(R; \{v_1, \ldots, v_k\})$  for some tree $R$.
\end{lem}

\begin{proof}{
We prove the assertion by induction on $k$. First assume that $k=1$. For convenience in notation, let   $v=v_1$, $s=s_1$,    and $\alpha=\alpha_1$.
We have $\varphi(T')=(x^2-1)^{s-1}((x^2-1)\varphi(T)-sx\varphi(T-v))$ by   Lemma  \ref{join}. Hence, we deduce from \eqref{ass}   that  $(x^2-1)\varphi(T)-sx\varphi(T-v)=f(x)(x^2-\alpha(s))$  for any integer $s\geqslant t$. In particular,
\begin{equation}\label{11}
(x^2-1)\varphi(T)-tx\varphi(T-v)=f(x)\big(x^2-\alpha(t)\big)
\end{equation}
and
\begin{equation}\label{12}
(x^2-1)\varphi(T)-(t+1)x\varphi(T-v)=f(x)\big(x^2-\alpha(t+1)\big).
\end{equation}
Using  \eqref{11} and \eqref{12}, one  obtains  that $f(x)=x\varphi(T-v)/(\alpha(t+1)-\alpha(t))$. It is clear from  \eqref{ass} that $f(x)$ is a monic polynomial,  implying  $\alpha(t+1)-\alpha(t)=1$. Therefore,     $f(x)=x\varphi(T-v)$.  It follows from \eqref{11} that $(x^2-1)\varphi(T)=x(x^2-\mu)\varphi(T-v)$  for some real number  $\mu$. Hence,  ${\rm mult}(T; 0)={\rm mult}(T-v;  0)+1$  and so it follows from    Lemma \ref{=}   that  $v$ is not adjacent to a vertex of degree $1$ in $T$. Consequently, $T$ contains $S(r)$ as an induced subgraph with the central vertex $v$,  where  $r$ is the degree of $v$. We know that the sum of squares of all eigenvalues of a graph equals twice the number of edges of the graph \cite[Proposition 1.3.1]{bh}. Applying  this fact to $T$ and $T-v$, we obtain that   $r=\mu-1$. This means that
$\lambda_1(T)=\lambda_1(S(r))$ and so   Theorem \ref{int} yields that  $T=S(r)$,  as desired.

Now assume that $k\geqslant2$. Let  $T''=T(v_1, \ldots, v_{k-1}; s_1, \ldots, s_{k-1})$. By  Lemma  \ref{join}, we have \begin{equation}\label{extra}\varphi(T')=(x^2-1)^{s_k-1}\big((x^2-1)\varphi(T'')-s_kx\varphi(T''-v_k)\big)
\end{equation}
Combining  \eqref{ass} with \eqref{extra} and   setting  $\rho=s_1+\cdots+s_{k-1}-k+1$, we conclude  that
$$(x^2-1)\varphi(T'')-s_kx\varphi(T''-v_k)=(x^2-1)^{\rho}f(x)\prod_{i=1}^k\big(x^2-\alpha_i(s_1, \ldots, s_k)\big),$$  for every  integers $s_1, \ldots, s_k\geqslant t$.
In particular, we have
\begin{equation}\label{k1}
(x^2-1)\varphi(T'')-tx\varphi(T''-v_k)=(x^2-1)^{\rho}f(x)\prod_{i=1}^k\big(x^2-\alpha_i(s_1, \ldots, s_{k-1},  t)\big)
\end{equation}
and
\begin{equation}\label{k2}
(x^2-1)\varphi(T'')-(t+1)x\varphi(T''-v_k)=(x^2-1)^{\rho}f(x)\prod_{i=1}^k\big(x^2-\alpha_i(s_1, \ldots, s_{k-1},  t+1)\big),
\end{equation}
for every  integers $s_1, \ldots, s_{k-1}\geqslant t$.
It is easily obtained from \eqref{k1} and \eqref{k2}    that
\begin{eqnarray}\label{es-1}
x\varphi(T''-v_k)=(x^2-1)^{\rho}f(x)\prod_{i=1}^{k-1}\big(x^2-\beta_i(s_1, \ldots, s_{k-1})\big),
\end{eqnarray}
where  $\beta_i(s_1, \ldots, s_{k-1})$, $\ldots$, $\beta_{k-1}(s_1, \ldots, s_{k-1})$ are   positive-valued function in terms  of $s_1, \ldots, s_{k-1}$.
It follows from  Remark \ref{weyl} that    $k-1$ of the roots of $\varphi(T''-v_k)$ tend to infinity as $s_1,\ldots,s_{k-1}$ grow and hence
$\prod_{i=1}^{k-1}(x^2-\beta_i(s_1, \ldots, s_{k-1}))$ is not divisible by $x$  for some integers      $s_1, \ldots, s_{k-1}$. So,   we find  from  \eqref{es-1}  that    $f(x)=xg(x)$ for some polynomial $g(x)$  and thus  we can rewrite \eqref{es-1} as
\begin{eqnarray}\label{es}
\varphi(T''-v_k)=(x^2-1)^{\rho}g(x)\prod_{i=1}^{k-1}\big(x^2-\beta_i(s_1, \ldots, s_{k-1})\big).
\end{eqnarray}

Let $W=\{v_1, \ldots, v_k\}$. By    \eqref{es},   Lemma \ref{delp2}, and  the induction hypothesis, we     deduce  that  each  connected component
$H$ of $T-v_k$  with $V(H)\cap W\neq\varnothing$  is of the form $\mathcal{S}(R; V(H)\cap W)$.
By replacing $v_k$  with any   of  $v_1, \ldots, v_{k-1}$, we find that this  property also  holds for    $T-v_1, \ldots, T-v_{k-1}$.
From this,  we  conclude   that  each  connected component
$H$ of $T-v_k$  with $V(H)\cap W=\varnothing$  must be  $P_2$, since if not, the connected component $H$ of  $T-v_i$   containing  $v_k$  does not have the form $\mathcal{S}(R; V(H)\cap W)$ for any $i\neq k$, a contradiction.

Denote   by  $L_1=\mathcal{S}(F_1;  A_1), \ldots, L_\ell=\mathcal{S}(F_\ell;  A_\ell)$    the    connected components  of $T-v_k$  which are not $P_2$.   In order to complete the proof, it is  clearly enough   to show   that the neighbor of $v_k$ in $V(L_i)$ is contained in $B_i=V(F_i)\setminus A_i$  for   $i=1,  \ldots, \ell$.
If $k=2$, then  $f(0)=0$ and  \eqref{ass} imply  that   $T'$ has eigenvalue $0$ and so  Corollary \ref{pm} yields   that  $T'$ and  $T$ have  no perfect matching. This     forces  that  the neighbor of $v_2$ in $V(L_1)$ to  be  contained in $B_1$. If  $k\geqslant3$ and  the neighbor of $v_k$ in $V(L_i)$ is not   contained in $B_i$ for  some $i$, then   the    connected component $H$ of $T-v_j$ containing   $v_k$  does not have the form $\mathcal{S}(R; V(H)\cap W)$ for any $j$ with $A_i\neq\{v_j\}$. This   completes the proof.
}\end{proof}

The following lemma is a special case of
\cite[Theorem 8.1.7]{crs}.

\begin{lem}\label{crslem}
Let    $e$ be  an edge of a tree  $T$.
Let $T'$ be the tree obtained from $T$ by contracting $e$ to a vertex $u$ and attaching  a pendant  vertex to $u$. Then
$\lambda_1(T')\geqslant\lambda_1(T)$.
\end{lem}

\begin{lem}\label{d+m+1}
Let $T$ be a tree of order $n$ and   $v\in V(T)$ be of  the degree $k$. For any positive integer $m$, define  $T_v(m)$ as  the tree obtained from $T$ by attaching  $m$ pendant vertices to $v$. Then
$\lambda_1^2(T_v(m))<m+k+1$ if   $m>(k+1)(n-k-2)$.
\end{lem}

\begin{proof}{
By applying  the operation described in Lemma \ref{crslem} iteratively on all the edges  of $T_v(m)$ not incident with  $v$, we reach at a   tree $T'_v(m)$ indicated in Figure \ref{af}.
\begin{figure}[H]
\centering
\includegraphics[width=0.5\textwidth]{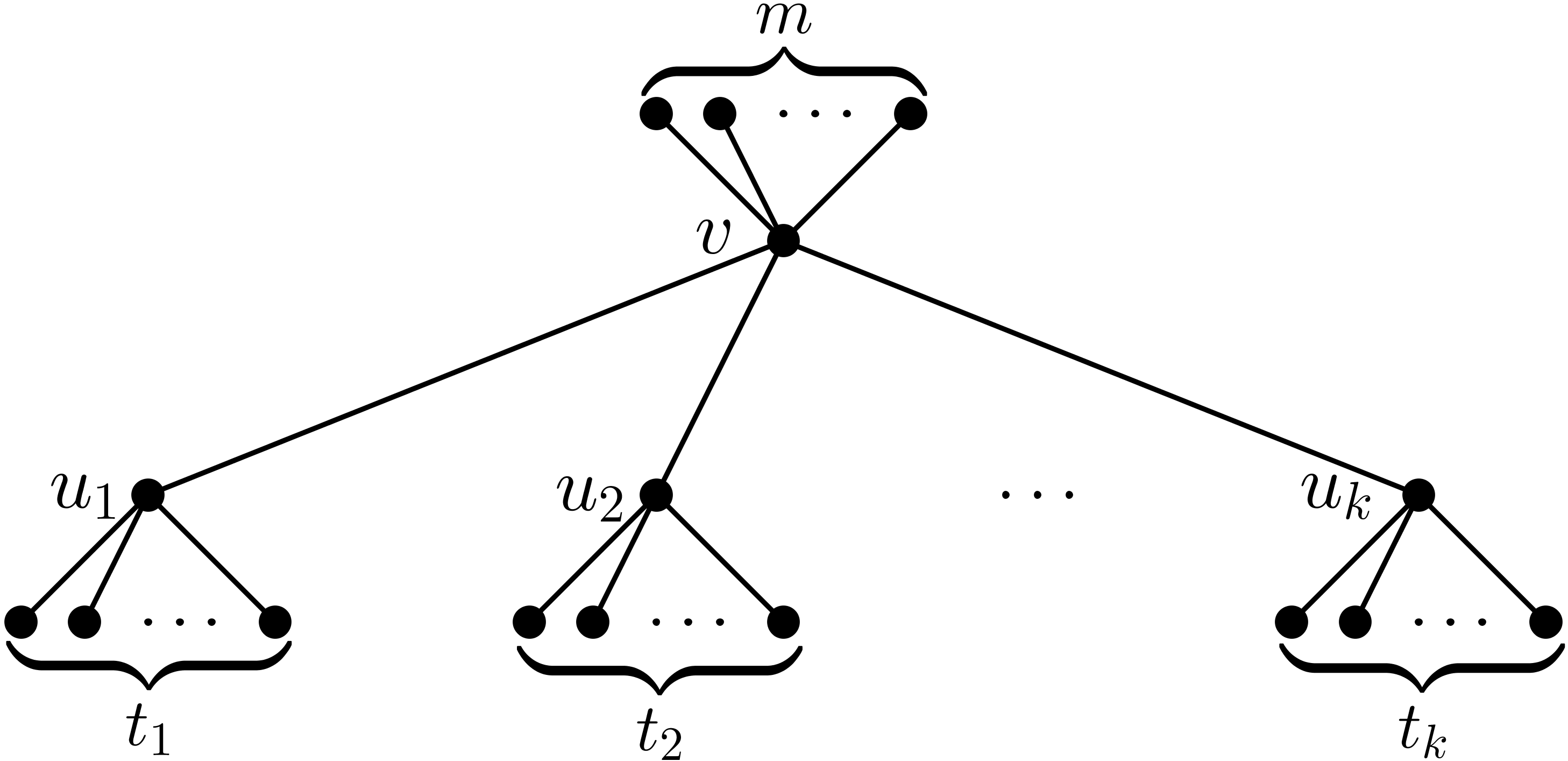}
\caption{\small{The tree $T'_v(m)$.}}\label{af}
\end{figure}
\noindent It follows from  Lemma \ref{crslem} and Theorem \ref{int} that  $\lambda_1(T_v(m))\leqslant\lambda_1(T'_v(m))\leqslant\lambda_1(T''_v(m))$,  where
$T''_v(m)$ is the tree  obtained for $T'_v(m)$ by increasing the number of pendant vertices attached
to each of $u_1, \ldots, u_k$ to $t=\max\{t_1, \ldots, t_k\}$.
The characteristic polynomial of $T''_v(m)$ can be computed by  applying  Lemma \ref{join}. So, an easy calculation shows that
$$\lambda_1^2\big(T''_v(m)\big)=\frac{m+k+t+\sqrt{(m+k+t)^2-4mt}}{2}.$$  Hence,  $\lambda_1^2(T_v(m))<m+k+1$  if $m>(k+1)(t-1)$. Since
$n\geqslant  k+t+1$, the  result   follows.
}\end{proof}

Now we are in a position to present  our main result.

\begin{thm}\label{finiteint}
For every integer $h\geqslant2$, there are  finitely many integral   trees with nullity  $h$.
\end{thm}

\begin{proof}{
Arguing toward a contradiction,  suppose that there are infinitely many  integral trees with nullity $h$ for some $h\geqslant2$.
By Theorem \ref{finitered}, there is a  tree $T$ with  $V(T)=\{v_1, \ldots, v_n\}$  such that
$T(v_1, \ldots, v_n; s_{i1}, \ldots, s_{in})$ is integral for an infinite set  $\{(s_{i1}, \ldots, s_{in})\}_{i\in \mathbb{N}}$ of $n$-tuples of nonnegative integers.
If for some fixed integers   $j$ and $s$, the set     $\{i\, |\, s_{ij}=s\}$ is infinite, then we replace $T$ by $T(v_j; s)$.
Repeating this  operation,  we may assume that there is a tree $T$ of order $n$  with specified  vertices $v_1, \ldots, v_k$
and an infinite set  $\{(s_{i1}, \ldots, s_{ik})\}_{i \, \in \,  \mathbb{N}}$ of $k$-tuples of nonnegative integers such that
$s_{ij}<s_{(i+1)j}$   for  $j=1, \ldots, k$,  and
$T_i=T(v_1, \ldots, v_k; s_{i1}, \ldots, s_{ik})$ is integral for all $i$.

By Remark \ref{weyl}, the set  $\{\lambda_j(T_i)\, |\, i\in \mathbb{N}\}$ is  not bounded for $j=1, \ldots,   k$,  and    by     Theorem \ref{int},  the set $\{\lambda_{k+1}(T_i)\, |\, i\in \mathbb{N}\}$ is  bounded above by  $\max\{1,\lambda_1(T-\{v_1, \ldots, v_k\})\}$. This clearly implies that there exists an integer $i_0$ such that $\lambda_j(T_i)$ is fixed for
$j=k+1, \ldots, k+\tfrac{n-h}{2}$ and  each   $i\geqslant i_0$.  Furthermore, by Lemma \ref{delp2}, we have  $\lambda_j(T_i)=1$ for $j=k+\tfrac{n-h}{2}+1, \ldots, s_{i1}+\cdots+s_{ik}+\tfrac{n-h}{2}$ and  all  $i\geqslant i_0$. By Theorem \ref{int}, it is not hard to see that
$T'=T(v_1, \ldots, v_k; s_1, \ldots, s_k)$ satisfies in \eqref{ass} for all integers $s_1, \ldots, s_k\geqslant t$, where $t=\max\{s_{i_01}, \ldots, s_{i_0k}\}$.
Therefore, it follows from  Lemma \ref{v1v2vk} that   $T$
has the form $\mathcal{S}(R; \{v_1, \ldots, v_k\})$ for some tree $R$.

We proceed to obtain a contradiction by showing  that  for large enough $i$, $\lambda_1(T_i)$ is not an integer.
For a fixed $i$,  we may relabel $v_1, \ldots, v_k$ such that   $s_{i1}\geqslant\cdots\geqslant s_{ik}$.
Using   Lemma \ref{cskvarithm}  twice and  by   Theorem \ref{int},  we find  that $\lambda_1^2(T_i)\leqslant 1+s_{i1}+\lambda_1^2(R)$.
Since $T_i$ contains     vertex disjoint copies of $S(s_{i1})$ and $S(s_{i2})$,        Theorem \ref{int} and  \eqref{sp}  imply  that  $\lambda_2^2(T_i)\geqslant \lambda_1^2(S(s_{i2}))=1+s_{i2}$. It follows that $\lambda_1^2(T_i)-\lambda_2^2(T_i)-\lambda_1^2(R)\leqslant  s_{i1}-s_{i2}$.
Since  $\lambda_1^2(T_i)-\lambda_2^2(T_i)$ is the difference of two distinct perfect squares,  $\lambda_1^2(T_i)-\lambda_2^2(T_i)$ and so  $ s_{i1}- s_{i2}$ tend to infinity    when   $i$ grows.
Further,      using   Lemma \ref{cskvarithm}  twice and   by  Theorem \ref{int}, we find that
\begin{equation}\label{inf1}
\lambda_1^2(T_i)\leqslant  1+s_{i2}+\lambda_1^2\big(R_{v_1}(s_{i1}-s_{i2})\big),
\end{equation}
where $R_{v_1}(s_{i1}-s_{i2})$ is as of Lemma \ref{d+m+1}.
Employing   Lemma \ref{d+m+1} and assuming $i$  is large enough, we obtain that
\begin{equation}\label{inf2}
\lambda_1^2\big(R_{v_1}(s_{i1}-s_{i2})\big)<s_{i1}-s_{i2}+\ell+1,
\end{equation}
where $\ell$ is  the degree of $v_1$ in  $R$.
Clearly, it follows from   $h\geqslant2$ and       Lemma \ref{=}  that   $k\geqslant2$.  From this and by    Theorem \ref{int}  and \eqref{sp}, one deduces that
\begin{equation}\label{inf3} \lambda_1^2(T_i)>\lambda_1^2\big(S(s_{i1}+\ell)\big)=s_{i1}+\ell+1.
\end{equation}
It follows from
\eqref{inf1}--\eqref{inf3} that $s_{i1}+\ell+1<\lambda_1^2(T_i)<s_{i1}+\ell+2$  for large enough
$i$.  This  contradiction  completes the proof.
}\end{proof}

\section{Integral trees with nullity $\mathbf{2}$ and $\mathbf{3}$}

Integral trees with nullity $0$ and $1$ are respectively  classified in   \cite{bro} and  \cite{wata0}.
In this  section  we characterize integral trees with nullity $2$ and $3$.
Before that,  we determine  all  integral trees among   the   trees introduced in Definition \ref{defs}.
From \eqref{sp}, we find  that $S(p)$ is integral if and only if $p+1$ is a perfect square.

\begin{thm}\label{inttr1}
Let $p$  and  $q$ be nonnegative integers. Then $S(p, q)$ is not integral.
\end{thm}

\begin{proof}{
Towards a contradiction,   suppose  that $T=S(p, q)$ is  integral.  We first assume  that  $p=q$.
Using   Lemma \ref{cskvarithm} twice, we find    that  $\lambda_1^2(T)=p+3$. Since $T$ has  two vertex disjoint copies of $S(p)$,  we obtain from  Theorem \ref{int} and \eqref{sp} that   $\lambda_2^2(T)\geqslant \lambda_1^2(S(p))=p+1$.  Therefore, $\lambda_1^2(T)-\lambda_2^2(T)\leqslant2$. This is a  contradiction,  since no   two distinct perfect squares have difference at most $2$.
We now  assume without loss of generality that $p>q$.
Again, using   Lemma \ref{cskvarithm}  twice and by     Theorem \ref{int}, we find that  $\lambda_1^2(T)<p+3$.
Since $T$ contains   a copy of $S(p+1)$ as a   subgraph,   Theorem \ref{int} and \eqref{sp} yield   that  $\lambda_1^2(T)>\lambda_1^2(S(p+1))=p+2$.
Hence, $p+2<\lambda_1^2(T)<p+3$ which implies that $\lambda_1(T)$ is not an integer, a contradiction.
}\end{proof}

\begin{thm}\label{inttr2}
Let $p$, $q$,      $r$ be nonnegative integers and let $T\in\{S(p, q, r), S'(p, q, r)\}$. Then either  $T=S(0, 0, 0)$ or  $T$ is not integral.
\end{thm}

\begin{proof}{
Assume   that $T$ is  integral and
let $t$ and $t'$ be the largest and  second largest number among $p, q, r$, respectively.   We know from \cite[Table 2]{cvet} that  $S(0, 0, 0)$ is    integral while $S'(0, 0, 0)$ is not integral. Hence, towards a contradiction,   we  suppose  that $t\geqslant1$.
Since $T$ contains  a copy of  $S(t+1)$ as a   subgraph,   Theorem \ref{int} and \eqref{sp}  imply  that   \begin{equation}\label{non1}
\lambda_1^2(T)>\lambda_1^2(S(t+1))=t+2.
\end{equation}
Assume that  $R$ is one of  the star graph of order  $4$ or $P_5$. Using Lemma \ref{cskvarithm} twice and by    Theorem \ref{int}, one  obtains  that  $\lambda_1^2(T)\leqslant 1+t+\lambda_1^2(R)$, where the equality occurs  if and only if  $p=q=r$.  From  \cite[pp. 8--9]{bro},    we find  that    $\lambda_1(R)=\sqrt{3}$ and so $\lambda_1^2(T)\leqslant t+4$.     In the case of equality, $T$ has three vertex disjoint copies of $S(t)$ and thus      Theorem \ref{int} and \eqref{sp}  yield  that  $\lambda_2^2(T)\geqslant\lambda_1^2(S(t))=t+1$ which in turn implies that  $\lambda_1^2(T)-\lambda_2^2(T)\leqslant3$. This   is impossible, since  $\lambda_1(T)$ and $\lambda_2(T)$ are two distinct  integers more than  $1$.
Thus,   in view of \eqref{non1}, one  deduces   that
$\lambda_1^2(T)=t+3$.
We know from \cite[Table 2]{cvet} that  $\lambda_1(S(1, 0, 0))$, $\lambda_1(S'(1, 0, 0))$, and $\lambda_1(S'(0, 1, 0))$  are greater than $2$.
This  implies  that    $t\geqslant2$ and therefore    $\lambda_1^2(T)-\lambda_2^2(T)\geqslant5$.     On the other hand, $T$ contains  two vertex disjoint copies of $S(t')$, so  Theorem \ref{int} and \eqref{sp} yield  that  $\lambda_2^2(T)\geqslant \lambda_1^2(S(t'))=t'+1$. This  follows that $t-t'\geqslant \lambda_1^2(T)-\lambda_2^2(T)-2\geqslant3$.
Assume that   $R'$ is one of the trees  $R'_1$, $R'_2$, $R'_3$    which are  depicted in Figure \ref{r}.
\begin{figure}[H]
\centering
\includegraphics[width=0.75\textwidth]{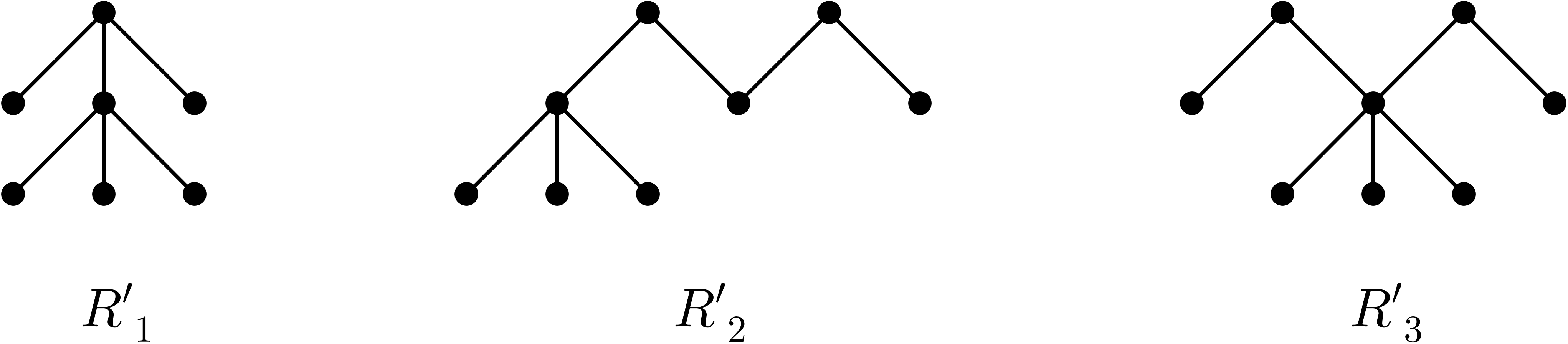}
\caption{\small{The tree $R'$.}}\label{r}
\end{figure}
\noindent By Theorem \ref{int} and using  Lemma \ref{cskvarithm} twice, one deduces  that  $t+3=\lambda_1^2(T)\leqslant 1+t-3+\lambda_1^2(R')$,  implying  $\lambda_1(R')\geqslant\sqrt{5}$. We know from \cite[Table 2]{cvet} that   $\lambda_1(R'_1)$ and $\lambda_1(R'_2)$    are less than $\sqrt{5}$. So,  $R'=R'_3$ and therefore  $T$ contains a copy of  $S(t+2)$  as   a   subgraph. Hence,  Theorem \ref{int} and \eqref{sp}  imply  that   $\lambda_1^2(T)>\lambda_1^2(S(t+2))=t+3$, a contradiction.
}\end{proof}

We are now ready to characterize integral trees with nullity $2$ and $3$. In order to do this in   a simple manner,   we use   the following interesting result  which is called the  Parter--Wiener theorem \cite{p, w}.

\begin{thm}   \label{parter} If $T$ is a tree and ${\rm mult}(T; \lambda)\geqslant2$ for some $\lambda$,  then there exists $v\in V(T)$   such that ${\rm mult}(T-v;  \lambda)={\rm mult}(T; \lambda)+1$.
\end{thm}

In the next  theorem, we  generalize an interesting result of  \cite{bro} by a short and simple proof.  We start with the following easy lemma.

\begin{lem}\label{tek}
Let  $T$ be  a tree  with      no eigenvalue in $(0, 1)\cup(1, 2)$.  Then the  order of  $T$ is at most $2\,{\rm mult}(T; 0)+4\,{\rm mult}(T; 1)-1$.
\end{lem}

\begin{proof}{
Since the spectrum of eigenvalues of $T$ is symmetric around the origin \cite[p. 6]{bh} and
the sum of squares of all eigenvalues of $T$ equals twice the number of its edges  \cite[Proposition 1.3.1]{bh},
we obtain that
$$4\big(n-{\rm mult}(T; 0)-2\,{\rm mult}(T; 1)\big)\leqslant2(n-1).$$
This follows the assertion.
}\end{proof}

\begin{thm}\label{nul1} Let  $T$ be  a tree  with  nullity $1$ and   no eigenvalue in $(0, 1)\cup(1, 2)$.  Then  $T=S(p)$ for some   $p\geqslant0$.
\end{thm}

\begin{proof}{
If ${\rm mult}(T; 1)\leqslant1$, then  Lemma \ref{tek} implies  that  $T$ is of order  at most   $5$.  We know from \cite[Table 2]{cvet} that, among the trees of order  at most five,  $S(0)$ is the only tree satisfying the assumption of the theorem. So, assume  that ${\rm mult}(T; 1)\geqslant2$. From  Theorem  \ref{parter}, there exists  a vertex  $v$  such that   ${\rm mult}(T-v; 1)={\rm mult}(T; 1)+1$. Hence,    Theorem \ref{int} implies that  $m(T-v)=0$. It follows  from Lemma  \ref{m0} that $T-v$ is a vertex disjoint union of some  copies  of $P_2$,  yielding the result.
}\end{proof}

The following conclusion,  which is first appeared in \cite{bro},  should be clear from  \eqref{sp} and  Theorem \ref{nul1}.

\begin{cor}
Each   integral  tree  with nullity $1$ is of the form $S(p^2-1)$ for some  $p\geqslant1$.
\end{cor}

\begin{thm}\label{nul2} Let  $T$ be a tree with  nullity $2$ and  no eigenvalue in $(0, 1)\cup(1, 2)$. Then  either $T$ is  the tree in depicted in  Figure \ref{y} or $T=S(p, q)$ for some nonnegative integers $p,  q$.
\end{thm}

\begin{proof}{
If   ${\rm mult}(T; 1)\leqslant1$, then  Lemma \ref{tek} yields that the order of $T$ is  at most $7$. We know from \cite[Table 2]{cvet} that,  among the trees of order  at most $7$,    the only tree satisfying the assumption of the theorem is the tree depicted in  Figure \ref{y}. So, assume  that ${\rm mult}(T; 1)\geqslant2$.  By Theorem  \ref{parter}, there exists a vertex $v$   such that ${\rm mult}(T-v; 1)={\rm mult}(T; 1)+1$.
Employing   Theorem \ref{int}, one concludes  that    $T-v$ has nullity $1$ and has   no eigenvalue in $(0, 1)\cup(1, 2)$.
Thus, in view of  Lemma  \ref{m0} and Theorem \ref{nul1},  $T-v$ is of the form   $S(p)\cup qP_2$   for some nonnegative integers  $p, q$.
If the neighbor of $v$ in $S(p)$ is not a vertex of degree $2$, then $T$ would have a perfect matching and so   Corollary \ref{pm} yields that  the   nullity of $T$  would be  $0$, a  contradiction. Hence,    $v$ is adjacent to a vertex of degree $2$ in $S(p)$. This means that $p\geqslant1$ and
$T=S(p-1, q)$, the result follows.
}\end{proof}

By combining  Theorem  \ref{inttr1} and  Theorem \ref{nul2}, the following is obtained.

\begin{cor}
There is only one integral tree with  nullity $2$;  namely,  the tree depicted in  Figure \ref{y}.
\begin{figure}[H]
\centering
\includegraphics[width=0.15\textwidth]{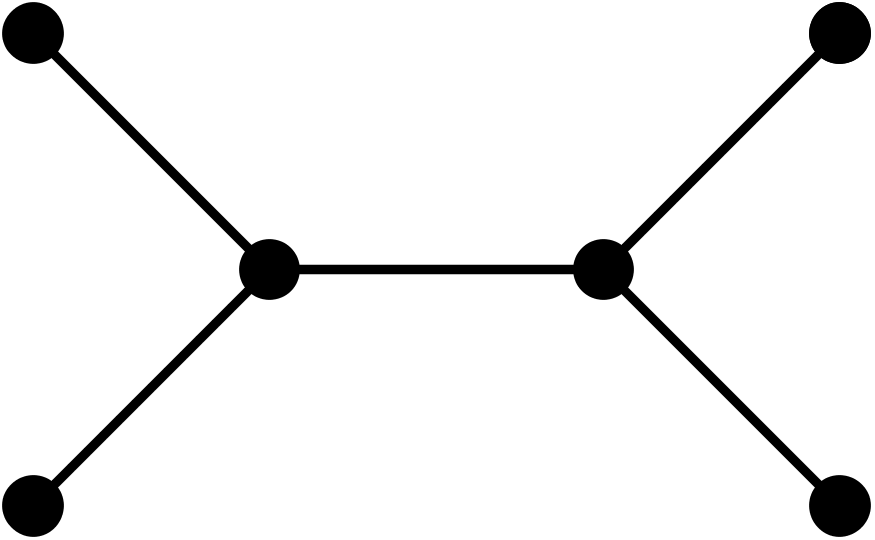}
\caption{The unique integral tree with nullity $2$.}\label{y}
\end{figure}
\end{cor}

\begin{thm}  The star graph  of order $5$  is the only  integral tree with  nullity $3$.
\end{thm}

\begin{proof}{
Let $T$ be an integral tree with nullity $3$.
If  ${\rm mult}(T; 1)\leqslant1$, then it follows  from   Lemma \ref{tek} that $T$ has at most $9$ vertices.
We know from \cite[Table 2]{cvet} that,  among the trees of order  at most $9$,  there is only one integral tree with nullity $3$ that is  the star graph of order $5$, we are done.
Towards a contradiction, suppose that   ${\rm mult}(T; 1)\geqslant2$.
From Theorem  \ref{parter}, there exists a vertex $v$  such that ${\rm mult}(T-v; 1)={\rm mult}(T; 1)+1$.
Moreover, by Theorem \ref{int}, $T-v$ has nullity $2$ and  has no eigenvalue in $(0,1)\cup(1,2)$.
It easily follows from Lemma \ref{=} that
$T-v$ has  no isolated vertex.
From Theorems \ref{nul1} and \ref{nul2}, it follows  for some nonnegative integers $p,  q, r$ that  $T-v$ is of one of the following forms:
\begin{itemize}
\item[(i)] $S(p)\cup S(q)\cup rP_2$;
\item[(ii)] $S(p, q)\cup rP_2$;
\item[(iii)] $Y\cup rP_2$, where $Y$ is the tree depicted in Figure \ref{y}.
\end{itemize}
If (i) is the case, then  by Corollary \ref{pm},   $v$ is necessarily adjacent to two vertices of degree $2$ in $S(p)$ and $S(q)$.
This means that $p, q\geqslant1$ and  $T=S'(p-1, r, q-1)$, which contradicts  Theorem \ref{inttr2}.
In the case (ii), it follows from  Corollary \ref{pm} that  the neighbor of  $v$ in $S(p, q)$ is adjacent    to a vertex of degree $1$.
This implies that  $T\in\{S(p, q, r), S'(r, p-1, q), S'(p, q-1, r)\}$ and so by   Theorem \ref{inttr2}, we find  that $T=S(0, 0, 0)$.  This is a contradiction, since
 ${\rm mult}(T;  1)\geqslant2$.
For the case (iii),   using  Corollary \ref{pm},  $v$ is necessarily adjacent to one of the two vertices of degree $3$ in $Y$. By  applying  Lemma \ref{join}, we   find that   $\varphi(T)=x^3(x^2-1)^r(x^4-(r+6)x^2+4r+6)$. From the intermediate value theorem,
it is easily seen that $\varphi(T)$ has a zero  in $(1, 2)$, a contradiction.
The proof is now complete.
}\end{proof}

We mention here that one can apply  a similar  method  to find all integral trees with other small  nullities  which  of course would be an elaborate task. By  \cite{bro}, among trees up to fifty
vertices,   there is no integral tree with nullities $4$, $6$,  or $9$. Therefore, one may ask  if there exist  integral trees with nullity  $4$. Further, one may ask  a more general question: Does  exist arbitrarily large integer $h$ such that there is no integral tree with nullity $h$?  Eventually, we pose the question: For given integers $m, k\geqslant1$, is the number of integral trees  with eigenvalue $m$ of  multiplicity $k$ finite?

\section*{Acknowledgments}

The authors are indebted to an anonymous referee for pointing out an error in the proof of the main result in an earlier version of the manuscript.
The authors thank both referees for their helpful comments and suggestions  which considerably improved the presentation of the article.
This research  was in part supported
by grants from IPM to  the first  author (No. 91050114) and   the second   author (No. 91050405).

{}

\end{document}